\documentclass[12pt]{article}
\usepackage{amsmath}
\usepackage{amssymb}
\usepackage{amsfonts}
\usepackage{amsthm}
\newtheorem{theorem}{Theorem}

\newtheorem{lemma}[theorem]{Lemma}

\newtheorem{proposition}[theorem]{Proposition}
\theoremstyle{definition}
\newtheorem{definition}[theorem]{\textbf{Definition}}
\newtheorem{remark}[theorem]{\textbf{Remark}}

\begin{document}

\date{ }
\title{Calibrated geodesic foliations \\
of the hyperbolic space}
\author{Yamile Godoy and Marcos Salvai~\thanks{%
Partially supported by \textsc{CONICET, FONCyT, SECyT~(UNC)}.} \\
%EndAName
{\small {FaMAF - CIEM}, Ciudad Universitaria, 5000 C\'{o}rdoba, Argentina}\\
ygodoy@famaf.unc.edu.ar, salvai@famaf.unc.edu.ar}

\maketitle

\begin{abstract} 
Let $H$ be the hyperbolic space of dimension $n+1$. A geodesic foliation of $%
H$ is given by a smooth unit vector field on $H$ all of whose integral
curves are geodesics. Each geodesic foliation of $H$ determines an $n$%
-dimensional submanifold $\mathcal{M}$ of the $2n$-dimensional manifold $%
\mathcal{L}$ of all the oriented geodesics of $H$ (up to orientation
preserving reparametrizations). The space $\mathcal{L}$ has a canonical
split semi-Riemannian metric induced by the Killing form of the isometry
group of $H$. Using a split special Lagrangian calibration, we study the
volume maximization problem for a certain class of geometrically
distinguished geodesic foliations, whose corresponding submanifolds of $%
\mathcal{L}$ are space-like.
\end{abstract}

\noindent Mathematics Subject Classification:\ 53C38, 53C12, 53C22, 53C50

\noindent Key words and phrases: split special Lagrangian calibration,
geodesic foliation, hyperbolic space, space of geodesics

\section{Introduction}

Calibrations are a tool to detect submanifolds of minimum volume in a
homology class. They originated in 1982 in the celebrated paper \cite%
{harvey-lawson} by Harvey-Lawson. Dadok and Morgan \cite{dadok,morgan}
obtained new important examples in 1983. Calibrations provided significant
achievements in volume minimization of submanifolds of the Euclidean space,
and they were also very useful in non-Euclidean settings, for instance in 
\cite{morgan-ziller,gluck-morgan-mackenzie,salvai01}. In 1989 Mealy \cite%
{Mealy} introduced calibrations on semi-Riemannian manifolds. The split
special Lagrangian calibrations were rediscovered by Warren \cite{Warren},
who applied them to the volume maximization problem of the special
Lagrangian submanifolds of a semi-Euclidean space (see also \cite{HarveyA}).
They also play a central role in the study of optimal
transportation \cite{warren-ot}.

In this article we will use calibrations in connection with geodesic
foliations. In this direction, Gluck and Ziller \cite{Gluck-Ziller}
calibrated the Hopf vector fields on the three sphere, which determine the
Hopf fibrations, that is, the canonical geodesic foliations of the three
sphere. We will deal with the hyperbolic case, but instead of calibrating
the vector field determining the foliation we will calibrate the foliation
itself thought of as the space of its leaves. Since the geodesics of the
hyperbolic space may be identify with Lorentzian planes in a Minkowski
space, the subject of this note is related also with the articles \cite%
{morgan-ziller,gluck-morgan-mackenzie}, dealing with the volume of
submanifolds of Grassmannians. In Section \ref{SCali} we recall basic
properties of split special Lagrangian calibrations. The fact that our
submanifolds are non-compact makes the definition of volume maximization
somehow involved.

Let $H$ be the hyperbolic space of dimension $n+1$. A geodesic foliation of $%
H$ is given be a smooth unit vector field on $H$ all of whose integral
curves are geodesics. Let $\mathcal{L}$ be the space of all the oriented
geodesics of $H$ (up to orientation preserving reparametrizations), which is
a $2n$ dimensional manifold and has a canonical split semi-Riemannian metric
induced by the Killing form of the isometry group of $H$. A geodesic
foliation of $H$ may be identified with an $n$ dimensional submanifold of $%
\mathcal{L}$. We consider a geometrically distinguished class of geodesic
foliations of $H$, which we call t.e.r.\ geodesic foliations, whose
corresponding submanifolds of $\mathcal{L}$ are space-like. This can be
found in Section \ref{Sfol}.

Now we comment on the contents of the last section. In Theorem \ref{Teo2} we
prove that the geodesic foliation $\mathcal{F}_{o}$ orthogonal to a totally
geodesic hypersurface of $H$ is volume maximizing among all t.e.r.\ geodesic
foliations. This follows from Theorem \ref{Teo1}, where the associated
submanifold $\mathcal{M}_{o}$ of $\mathcal{L}$ is proved to be homologically
volume maximizing in an open submanifold of $\mathcal{L}$, using a split
special Lagrangian calibration.

\section{Calibrations\label{SCali}}

Let $N$ be a semi-Riemannian manifold. We will say that an $m$-vector $\xi $
in $T_{q}N$ is space-like if the subspace generated by $\xi $ is space-like
(we are not considering the induced indefinite inner product on $\Lambda
^{m}\left( T_{q}N\right) $). 

We will deal only with split semi-Riemannian manifolds. So, the definitions
bellow are given in this case. Specifically, $N$ will be a semi-Riemannian
manifold of dimension $2m$ with signature $\left( m,m\right) $.

\begin{definition}
Let $N$ be a split semi-Riemannian manifold of dimension $2m$. A closed $m$%
-form $\psi $ on $N$ is called a \emph{calibration} if $\psi _{q}\left( \xi
\right) \geq $ vol$_{m}\left( \xi \right) $ for any oriented space-like $m$%
-vector $\xi $ in $T_{q}N$ with $\psi _{q}(\xi )>0$, for all $q\in N$.

Let $\psi $ be a calibration on $N$ and let $M$ be an oriented space-like
submanifold of $N$ of dimension $m$. Then $M$ is said to be \emph{calibrated
by }$\psi $ if $\psi _{q}\left( \xi \right) =$ vol$_{m}\left( \xi
\right) $ for any $m$-vector $\xi $ generating $T_{q}M$ with $\psi _{q}(\xi
)>0$, for all $q\in M$.
\end{definition}

\begin{definition}
\label{hvm} a) Let $M$ be an oriented $m$-dimensional space-like submanifold of $%
N$. One says that $M$ is \emph{volume maximizing in }$N$ if for any open
subset $U$ of $M$ with compact closure and smooth border $\partial U$, one
has that 
\begin{equation}
\text{vol}_{m}\left( U\right) \geq \text{vol}_{m}\left( V\right)
\label{volUvolV}
\end{equation}%
for any space-like submanifold $V$ of $N$ of dimension $m$ with compact
closure and $\partial V=\partial U$. 

b) Moreover, $M$ is said to be \emph{homologically volume maximizing in }$N$ 
if in addition $V$ is required to be homologous to $U$.

In this context, uniqueness of $M$ is understood as follows: equality holds
in (\ref{volUvolV}) only if $U=V$.
\end{definition}

We recall the fundamental theorem of calibrations introduced by Mealy, which
is analogous to the corresponding result for Riemannian manifolds in \cite%
{harvey-lawson}.

\begin{theorem}
\label{tfc}\emph{\cite{Mealy}} Let $N$ be a split semi-Riemannian manifold
of dimension $2m$ and let $\psi $ be a calibration on $N$. If $M$ is an
oriented space-like submanifold of $N$ of dimension $m$ which is calibrated
by $\psi $, then $M$ is homologically volume maximizing in $N$.
\end{theorem}

Next we introduce the split special Lagrangian calibration on a split
Euclidean space, which appeared first in the work of Mealy \cite{Mealy}. We
consider the presentation of this calibration in null coordinates given in 
\cite{Warren}.

\begin{proposition}
\label{warren}\emph{\cite{Warren}} Let $\mathbb{R}^{n,n}$ be the product $\mathbb{R}%
^{n}\times \mathbb{R}^{n}$ endowed with the split inner product whose
associated square norm is $\left\Vert \left( x,y\right) \right\Vert
=\left\langle x,y\right\rangle $, where $\left\langle .,.\right\rangle $ is
the canonical inner product on $\mathbb{R}^{n}$. For any $c>0$, the $n$-form%
\begin{equation*}
\phi _{c}=\tfrac{1}{2}\left( ce^{1}\wedge \dots \wedge e^{n}+\tfrac{1}{c}%
f^{1}\wedge \dots \wedge f^{n}\right) 
\end{equation*}%
is a calibration on $\mathbb{R}^{n,n}$ \emph{(}here $\left\{ e^{i}\mid
i=1,\dots ,n\right\} $ and $\left\{ f^{j}\mid j=1,\dots ,n\right\} $ are the
dual of the canonical bases of $\mathbb{R}^{n}\times \left\{ 0\right\} $ and 
$\left\{ 0\right\} \times \mathbb{R}^{n}$, respectively\emph{)}. Moreover, a
space-like $n$-vector $\xi $ is calibrated by $\phi _{c}$ if and only if $%
\xi $ is special Lagrangian, that is, if%
\begin{equation*}
f^{1}\wedge \dots \wedge f^{n}\left( \xi \right) =c^{2}e^{1}\wedge \dots
\wedge e^{n}\left( \xi \right) \text{.}
\end{equation*}
\end{proposition}

\section{Geodesic foliations of the hyperbolic space\label{Sfol}}

\subsection{The space of oriented geodesics of $H$}

Let $H$ be the hyperbolic space of constant sectional curvature $-1$ and
dimension $n+1$. We recall from \cite{Salvai-2007} (see also \cite{Georgiou-2010}) some facts about
the geometry of the space $\mathcal{L}$ of all complete oriented geodesics
of $H$ (up to orientation preserving reparametrizations). It admits a unique
differentiable structure of dimension $2n$ such that the canonical
projection $\pi :T^{1}H\rightarrow \mathcal{L}$ is a smooth
submersion.

Let $\gamma :\mathbb{R}\rightarrow H$ be a unit speed geodesic and let $%
\mathcal{J}_{\gamma }$ be the space of all Jacobi vector fields along $%
\gamma $ which are orthogonal to $\dot{\gamma}$. There exists a well-defined
canonical isomorphism 
\begin{equation}
T_{\gamma }:\mathcal{J}_{\gamma }\rightarrow T_{[\gamma ]}\mathcal{\mathcal{L%
}}\text{,}\hspace{1cm}T_{\gamma }(J)=\left. {\frac{d}{dt}}\right\vert
_{0}[\gamma _{t}]\text{,}  \label{isoT}
\end{equation}%
where $\gamma _{t}$ is any variation of $\gamma $ by unit speed geodesics
associated with $J$.

Given a tangent vector $X$ to a semi-Riemannian manifold, we denote $%
\left\Vert X\right\Vert =\left\langle X,X\right\rangle $ and $\left\vert
X\right\vert =\sqrt{\left\vert \left\langle X,X\right\rangle \right\vert }$.
Also, given $v\in TH$, we denote by $\gamma _{v}$ the unique geodesic in $H$
with initial velocity $v$.

Let $G$ be the identity component of the isometry group of $H$. The manifold 
$\mathcal{L}$ is a homogeneous space of $G$ and the Killing form induces on
it a semi-Riemannian metric $g$ of signature $\left( n,n\right) $. In terms
of the isomorphism (\ref{isoT}) the square norm of this metric may be
written as follows: For $J\in \mathcal{J}_{\gamma }$, 
\begin{equation}
\Vert T_{\gamma }(J)\Vert =\left\vert J\right\vert ^{2}-\left\vert J^{\prime
}\right\vert ^{2}\text{,}  \label{n}
\end{equation}%
where $J^{\prime }$ denotes the covariant derivative of $J$ along $\gamma $
(the right hand side in the expression is a constant function, so the left hand side is well
defined). For the geometrical meaning of a curve in $\left( \mathcal{L}%
,g\right) $ being space- or time-like see Proposition 4 in \cite{Salvai-2007}.

It will be convenient for us to have another presentation of $\left( \mathcal{%
L},g\right) $. If one considers the unit ball model of $H$, given an
oriented geodesic $\gamma $ of $H$, then $\gamma \left( \infty \right) $ and 
$\gamma \left( -\infty \right) $ in $S^{n}$ are well defined. Let $\Delta
_{n}=\left\{ \left( p,p\right) \mid p\in S^{n}\right\} $ denote the diagonal
in $S^{n}\times S^{n}$. The map 
\begin{equation}
F:\mathcal{L}\rightarrow \left( S^{n}\times S^{n}\right) -\Delta _{n}\text{%
,\ \ \ \ \ }F\left( \left[ \gamma \right] \right) =\left( \gamma \left(
-\infty \right) ,\gamma \left( \infty \right) \right)   \label{isoF}
\end{equation}%
is a well-defined diffeomorphism.

Given distinct points $p,q\in S^{n}$, let $T_{p,q}$ denote the reflection on 
$\mathbb{R}^{n+1}$ with respect to the hyperplane orthogonal to $p-q$. Let
us consider on $\left( S^{n}\times S^{n}\right) -\Delta _{n}$ the
semi-Riemannian metric whose associated norm is 
\begin{equation}
\left\Vert \left( x,y\right) \right\Vert _{\left( p,q\right) }=4\left\langle
T_{p,q}x,y\right\rangle /\left\vert q-p\right\vert ^{2}  \label{mss}
\end{equation}%
for $x\in p^{\bot }=T_{p}S^{n}$, $y\in q^{\bot }=T_{q}S^{n}$. Then, if $\mathcal{L}$ is endowed
with the metric $g$, the diffeomorphism $F$ given in (\ref{isoF}) is an
isometry.

\subsection{Geodesic foliations}

A smooth geodesic foliation of $H$ is given by a smooth unit vector field $V$
on $H$ all of whose integral curves, the leaves, are geodesics. In \cite{Godoy-Salvai14} we
studied geodesic foliations of the three dimensional hyperbolic space. The
set $\mathcal{F}$ of all the leaves admits a canonical differentiable
structure such that the canonical projection $H\rightarrow \mathcal{F}$ is a
smooth submersion and is naturally embedded in $\mathcal{L}$. We will denote
by $\mathcal{M}$ the associated submanifold of $\mathcal{L}$, of dimension $n
$. The space of oriented geodesics of the \emph{three dimensional}
hyperbolic space admits another canonical split semi-Riemannian metric,
apart from that induced by the Killing form. We characterized geometric
properties of a geodesic foliation $\mathcal{F}$ in terms of the
nondegeneracy of the induced metrics on the corresponding submanifold $%
\mathcal{M}$ of $\mathcal{L}$.

Now, we consider a new type of geodesic foliations of $H$, for any dimension. Their informal
geometrical meaning is that the \emph{translational} motion of the leaves 
\emph{exceeds} their \emph{rotational} motion (a curve of geodesics in the
foliation may be thought of as a motion in $H$ of the initial geodesic), and
so we call them t.e.r.\ geodesic foliations.

\begin{definition}
A smooth geodesic foliation $\mathcal{F}$ determined by the unit vector
field $V$ on $H$ is a \emph{t.e.r.\ geodesic foliation} if $\left\vert
\nabla _{v}V\right\vert <\left\vert v\right\vert $, for any nonzero tangent
vector $v$ of $H$.
\end{definition}

Standard examples of t.e.r.\ geodesic foliations are the foliations which
are orthogonal to totally geodesic submanifolds of codimension one of $H$.

\begin{proposition}
\label{rlt}Let $\mathcal{F}$ be a smooth geodesic foliation of $H$. Then $%
\mathcal{F}$ is a t.e.r.\ geodesic foliation if and only if the associated
submanifold $\mathcal{M}$ of $\left( \mathcal{L},g\right) $ is space-like.
\end{proposition}

%\noindent \textbf{Proof.} 
\begin{proof} Since $\nabla _{V}V=0$, we may suppose that $v$ is
orthogonal to $V$. Given a submanifold $\mathcal{M}$ of $\mathcal{L}$ and $%
[\gamma ]\in \mathcal{M}$, any nonzero tangent vector $X$ in $T_{[\gamma ]}%
\mathcal{M}$ corresponds via $T_{\gamma }$ (\ref{isoT}) to a Jacobi vector
field in $\mathcal{J}_{\gamma }$ associated with a variation of $\gamma $ by
unit speed geodesics whose equivalence classes are in $\mathcal{M}$ (see
\cite{Godoy-Salvai14}). That is, there exists a smooth curve $c:(-\varepsilon ,\varepsilon
)\rightarrow H$ with $c^{\prime }(0)\neq 0$ (since $X\neq 0$) such that $%
X=T_{\gamma }(J)$, where $J(s)=\left. \frac{d}{dt}\right\vert _{0}\gamma
_{V(c(t))}\left( s\right) $. Now, since $J^{\prime }(0)=\nabla _{J(0)}V$, we
have by (\ref{n}) that%
\begin{equation*}
\Vert X\Vert =|J(0)|^{2}-|J^{\prime }(0)|^{2}=|c^{\prime }(0)|^{2}-|\nabla
_{J(0)}V|^{2}\text{.}
\end{equation*}%
Therefore, the equivalence holds. 
\end{proof}
%\hfill $\square $

%\medskip

Let $\varphi _{\pm }:\mathcal{L}\rightarrow S^{n}$ be the forward (for +)
and backward (for $-$) Gauss maps, that is, $\varphi _{\pm }([\gamma
])=\gamma (\pm \infty )$.

\begin{proposition}
\label{rltdifeo}If $\mathcal{F}$ is a t.e.r.\ geodesic foliation of $H$ and $%
\mathcal{M}$ is the associated submanifold of $\mathcal{L}$, then $\varphi
_{\pm }:\mathcal{M}\rightarrow S^{n}$ are local diffeomorphisms.
\end{proposition}

The proof is similar to the proof of the converse in Theorem 4.3 (a) in \cite{Godoy-Salvai14},
where Lemma 4.5 in that paper is used (which holds for higher dimensions).

\section{Calibrated geodesic foliations}

In this section we state and prove the main results. 

\begin{proposition}
\label{F0sl}Let $S$ be a complete totally geodesic submanifold of $H$ of codimension one and let $V
$ be a unit vector field normal to $S$. Then 
\begin{equation*}
\phi :S\rightarrow \mathcal{L},\hspace{0.5cm}\ \ \ \ \phi \left( q\right)
=[\gamma _{V\left( q\right) }]\text{,}
\end{equation*}%
is a space-like submanifold of $(\mathcal{L},g)$.
\end{proposition}

%\textbf{Proof. }
\begin{proof} First we observe that $\phi =\pi \circ V$, so $\phi $ is
smooth ($\pi $ is the canonical projection $T^{1}H\rightarrow \mathcal{L}$).
Also, it is one to one since $V$ is normal to $S$. Let us see that $(d\phi
)_{p}$ is injective for any $p\in S$. Let $w\in T_{p}S$ such that $(d\phi
)_{p}w=0$ and let $\alpha :(-\varepsilon ,\varepsilon )\rightarrow S$ be a
smooth curve with $\alpha ^{\prime }(0)=w$. We compute 
\begin{equation*}
0=(d\phi )_{p}w=\left. {\frac{d}{dt}}\right\vert _{0}\pi \circ V(\alpha
(t))=T_{\gamma _{V(p)}}(J),
\end{equation*}%
where $T_{\gamma _{V(p)}}$ is the isomorphism given in (\ref{isoT}) and $%
J(s)=\left. {\frac{d}{dt}}\right\vert _{0}\gamma _{V(\alpha (t))}(s)$ is the
associated Jacobi vector field in $\mathcal{J}_{\gamma _{V(p)}}$ satisfying $%
J(0)=w$. Then $J\equiv 0$ and so $w=0$, as desired.

Finally, we check that $\Vert (d\phi )_{p}w\Vert >0$ for any nonzero $w\in
T_{p}S$. By (\ref{n}) and the computation above, 
\begin{equation*}
\Vert (d\phi )_{p}w\Vert =|J(0)|^{2}-|J^{\prime }(0)|^{2}=|w|^{2}-|J^{\prime
}(0)|^{2}.
\end{equation*}%
Since $J^{\prime }(0)=\left. {\frac{D}{dt}}\right\vert _{0}V(\alpha (t))$
and $V$ is a parallel vector field along $\alpha $ ($S$ is a totally
geodesic hypersurface), we have that $J^{\prime }(0)=0$ and so $\Vert (d\phi
)_{p}w\Vert =|w|^{2}>0$. 
\end{proof}
%\hfill $\square $

%\medskip

Now, we call $\mathcal{M}_{o}$ the image of $\phi $, that is $\mathcal{M}%
_{o}=\{[\gamma _{V\left( q\right) }]\in \mathcal{L} \mid q\in S\}$, and denote 
\begin{equation*}
\mathcal{L}^{\prime }=\left\{ \ell \in \mathcal{L}\mid \ell \left( \infty
\right) \in \varphi _{+}\left( \mathcal{M}_{o}\right) \text{ and }\ell
\left( -\infty \right) \in \varphi _{-}\left( \mathcal{M}_{o}\right)
\right\} \text{,}
\end{equation*}%
which is an open subset of $\mathcal{L}$.

The following theorem refers to Definition \ref{hvm} (b).

\begin{theorem}
\label{Teo1}The submanifold $\mathcal{M}_{o}$ is homologically volume
maximizing in $\mathcal{L}^{\prime }$ and $\mathcal{M}_{o}$ is unique with
this property.
\end{theorem}

%\noindent \textbf{Proof.} 
\begin{proof} We work with the Poincar\'{e} unit ball model of $H
$, included in $\mathbb{R}^{n+1}$. We may suppose without loss of generality
that $S=H\cap e_{0}^{\bot }$, where $\left\{ e_{0},\dots ,e_{n}\right\} $ is
the canonical basis of $\mathbb{R}^{n+1}$. Since in this model $S$ is an $n$%
-dimensional ball, we denote it by $B$ (to avoid confusion with the border $%
S^{n}$ of $H$).

By the fundamental theorem of calibrations (Theorem \ref{tfc}) it suffices
to find a calibration $\psi $ on $\mathcal{L}^{\prime }$ calibrating $%
\mathcal{M}_{o}$. In order to define $\psi $ it will be convenient to
consider the model of $\mathcal{L}$ given by the isometry $F$ in (\ref{isoF}%
). Via this isometry, $\mathcal{L}^{\prime }$ corresponds to $%
S_{-}^{n}\times S_{+}^{n}$, where $S_{\pm }^{n}=\left\{ p\in S^{n}\mid \pm
p_{0}>0\right\} $ are the upper and lower hemispheres ($p_{0}$ is the $0$-th
coordinate of $p$). Recall that the open set $S_{-}^{n}\times
S_{+}^{n}\subset \left( S^{n}\times S^{n}\right) -\Delta _{n}$ carries the
split semi-Riemannian metric (\ref{mss}) induced from the diffeomorphism $F$.

Let $\theta $ be the volume form on the sphere $S^{n}$ with the round metric
of constant curvature one. Let $\pi _{i}:\left( S^{n}\times S^{n}\right)
-\Delta _{n}\rightarrow S^{n}$ be the projection to the $i$-th factor ($i=1,2
$) and let $\psi $ be the $n$-form on $S_{-}^{n}\times S_{+}^{n}$ defined by 
\begin{equation*}
\psi _{\left( p,q\right) }=\frac{1}{2}\left( \pi _{1}^{\ast }\left( \frac{1}{%
\left\vert p_{0}\right\vert ^{n}}~\theta \right) +\pi _{2}^{\ast }\left( 
\frac{1}{\left\vert q_{0}\right\vert ^{n}}~\theta \right) \right) \text{.}
\end{equation*}%
Let us see that $\psi $ is a calibration calibrating $F\left( \mathcal{M}%
_{o}\right) $. Clearly, $\psi $ is closed, since each term is the pull-back
of an $n$-form on the $n$-dimensional manifold $S^{n}$.

Now we fix $\left( p,q\right) $ in $S_{-}^{n}\times S_{+}^{n}$ and apply
Proposition \ref{warren}. Let $\left\{ e_{1},\dots ,e_{n},f_{1},\dots
,f_{n}\right\} $ be the canonical basis of $\mathbb{R}^{n,n}$ as in that
proposition and let $A:\mathbb{R}^{n,n}\rightarrow T_{p}S^{n}\times
T_{q}S^{n}$ be the linear isomorphism given by 
\begin{equation*}
A\left( e_{i}\right) =\left( v_{i},0\right) \text{,\ \ \ \ \ \ \ }A\left(
f_{j}\right) =\left( 0,T_{p,q}v_{j}\right) \text{,}
\end{equation*}%
where $b=\left\{ v_{i}\mid 1\leq i\leq n\right\} $ is an orthogonal basis of 
$T_{p}S^{n}$ with $\left\vert v_{i}\right\vert =\left\vert q-p\right\vert /2$
and $T_{p,q}$ is the reflection defined in the paragraph above (\ref{mss}).
Then $A$ is an isometry (the inner product on $T_{p}S^{n}\times T_{q}S^{n}$
is as in (\ref{mss})). A straightforward computation yields%
\begin{equation*}
A^{\ast }\psi _{\left( p,q\right) }=C\phi _{c}
\end{equation*}%
with%
\begin{equation*}
C=\frac{\left\vert q-p\right\vert ^{n}}{2^{n}\left\vert q_{0}\right\vert
^{n/2}\left\vert p_{0}\right\vert ^{n/2}}\ \ \ \ \ \ \ \ \ \text{and\ \ \ \
\ \ \ \ \ }c=\frac{\left\vert q_{0}\right\vert ^{n/2}}{\left\vert
p_{0}\right\vert ^{n/2}}\text{.}
\end{equation*}%
Now we see that $C\geq 1$. We compute 
\begin{eqnarray*}
\left\vert q-p\right\vert ^{2} \geq \left( q_{0}-p_{0}\right)
^{2}=q_{0}^{2}-2q_{0}p_{0}+p_{0}^{2}=\left\vert q_{0}\right\vert ^{2}+2\left\vert q_{0}\right\vert \left\vert
p_{0}\right\vert +\left\vert p_{0}\right\vert ^{2}\text{,}
\end{eqnarray*}%
since $p_{0}<0$ and $q_{0}>0$. Thus,%
\begin{equation}\label{q-p}
\left\vert q-p\right\vert ^{2}-4\left\vert q_{0}\right\vert \left\vert
p_{0}\right\vert \geq \left( \left\vert q_{0}\right\vert -\left\vert
p_{0}\right\vert \right) ^{2}\geq 0\text{.}
\end{equation}
Hence, $\left\vert q-p\right\vert \geq 2\left\vert q_{0}\right\vert^{1/2} \left\vert
p_{0}\right\vert^{1/2}$. Consequently, Proposition \ref{warren} implies that $\psi $ is a calibration
on $S_{-}^{n}\times S_{+}^{n}$ and so $C\geq 1$.

Next, we verify that the corresponding calibration on $\mathcal{L}^{\prime }$
calibrates $\mathcal{M}_{o}$. We have that 
\begin{equation*}
F\left( \mathcal{M}_{o}\right) =\left\{ \left( p,Tp\right) \mid p\in
S_{-}^{n}\right\} \text{,}
\end{equation*}%
where $T:S^{n}\rightarrow S^{n}$ is defined by $T=T_{e_{0},-e_{0}}$, that
is,\ $T\left( p_{0},p_{1},\dots ,p_{n}\right) =\left( -p_{0},p_{1},\dots
,p_{n},\right) $. Therefore, 
\begin{equation*}
T_{\left( p,Tp\right) }F\left( \mathcal{M}_{o}\right) =\left\{ \left(
v,Tv\right) \mid v\in T_{p}S^{n}=p^\bot\right\} \text{.}
\end{equation*}
Now, by polarization of (\ref{mss}), for $v_{i}\in b$ as above, the set $\left\{
\left( v_{i},Tv_{i}\right) \mid 1\leq i\leq n\right\} $ is an orthonormal
basis of $T_{\left( p,Tp\right) }F\left( \mathcal{M}_{o}\right) $. We compute%
\begin{equation*}
\begin{array}{l}
\psi _{\left( p,Tp\right) }\left( \left( v_{1},Tv_{1}\right) \wedge \dots
\wedge \left( v_{n},Tv_{n}\right) \right) = \\ 
\ \ \ \ \ \ \ \ \ \ \ \ =\frac{1}{2}\left( \frac{1}{\left\vert
p_{0}\right\vert ^{n}}\theta _{p}\left( v_{1}\wedge \dots \wedge
v_{n}\right) +\frac{1}{\left\vert -p_{0}\right\vert ^{n}}\theta _{Tp}\left(
Tv_{1}\wedge \dots \wedge Tv_{n}\right) \right)  \\ 
\ \ \ \ \ \ \ \ \ \ \ \ =\frac{1}{2}\left( \frac{1}{\left\vert
p_{0}\right\vert ^{n}}\left( \frac{\left\vert p-Tp\right\vert }{2}\right)
^{n}+\frac{1}{\left\vert -p_{0}\right\vert ^{n}}\left( \frac{\left\vert
p-Tp\right\vert }{2}\right) ^{n}\right) =1%
\end{array}%
\end{equation*}%
since $T$ is an isometry and $\left\vert p-Tp\right\vert =2\left\vert
p_{0}\right\vert $.

To conclude, we show the uniqueness. Suppose that $\mathcal{U}$ and $%
\mathcal{V}$ are as in Definition \ref{hvm} with $M=\mathcal{M}_{o}$ and $N=%
\mathcal{L}^{\prime }$. We have%
\begin{equation*}
\text{vol}_{n}\left( \mathcal{V}\right) =\int_{\mathcal{V}}\text{vol}%
_{n}\leq \int_{\mathcal{V}}i^{\ast }\psi =\int_{\mathcal{U}}i^{\ast }\psi =%
\text{vol}_{n}\left( \mathcal{U}\right) 
\end{equation*}%
(here $i$ is the inclusion;\ the inequality holds since $\psi $ is a
calibration, the last and next to last equalities hold since $\mathcal{U}$
is calibrated by $\psi $ and $\mathcal{U}$ and $\mathcal{V}$ are homologous,
respectively). If vol$_{n}\left( \mathcal{V}\right) =$ vol$_{n}\left( 
\mathcal{U}\right) $, then vol$_{n}=i^{\ast }\psi $ on $\mathcal{V}$ and
this is only possible if $\mathcal{V}=\mathcal{U}$ since (\ref{q-p}) yields that $C>1$ outside $\mathcal{M}_{o}$.
\end{proof}

In the proof of the next theorem we will need the following elementary
topological lemma. We give the proof for the sake of completeness.

\begin{lemma}
\label{lemaBola}Let $B$ be as above the open ball of radius one centered at
the origin of $\mathbb{R}^{n+1}$ and let $h:B\rightarrow B$ be a local
homeomorphism such that $h$ coincides with the identity on the complement of
a closed ball contained in $B$. Then $h$ is a homeomorphism.
\end{lemma}

\begin{proof} Notice that $h$ extends continuously to the compact ball $%
\bar{B}$. The map $h$ is onto $B$ since otherwise the sphere $S^{n}$ would
be a deformation retract of $B$, a contradiction. Now we check that $h$ is
one to one. Since $B$ is simply connected, it suffices to see that $h$ is a
covering map. For $q\in B$, the preimage of $\left\{ q\right\} $ consists of
a finite number of points $\left\{ p_{1},\dots ,p_{m}\right\} $ (it can not
contain an accumulation point, since $h$ is locally one to one). Since $h$
is a local homeomorphism, there exist open neighborhoods $U_{1},\dots ,U_{m}$
and $V$ of $p_{1},\dots ,p_{m}$ and $q$, respectively, such that $\left.
h\right\vert _{U_{i}}:U_{i}\rightarrow V$ are homeomorphisms. It remains to
show that there exists an open neighborhood $W\subset V$ of $q$ satisfying
that $h^{-1}\left( W\right) \subset U_{1}\cup \dots \cup U_{m}$. For any $%
k\in \mathbb{N}$, let $B_{k}$ the ball of center $q$ and radius $1/k$. We
can take $W=B_{k_{o}}$ for some $k_{o}$. Indeed, if for each $k$ there
exists $x_{k}\in h^{-1}\left( B_{k}\right) $ with $x_{k}\notin U_{1}\cup
\dots \cup U_{m}$, an accumulation point of the sequence $x_{k}$ in $\bar{B}$
should be one of the $p_{i}$'s or be in the border of $\bar{B}$, which is a
contradiction. 
\end{proof}

Let $\mathcal{F}_{o}$ be the foliation of $H$ determined by the unit vector
field $V$ normal to $B$ considered at the beginning of this section. Since $%
\mathcal{M}_{o}$ is a space-like submanifold of $\mathcal{L}$, by
Proposition \ref{rlt} we know that $\mathcal{F}_{o}$ is a t.e.r.\ geodesic
foliation. By the volume of a t.e.r.\ geodesic foliation $\mathcal{F}$ we
understand the volume of the corresponding submanifold $\mathcal{M}$ of $%
\mathcal{L}$.

\begin{theorem}\label{Teo2} 
The foliation $\mathcal{F}_{o}$ is volume maximizing among all
t.e.r.\ geodesic foliations $\mathcal{F}$ of $H$. More precisely, as $%
\mathcal{M}_{o}$ is not compact, $\mathcal{F}_{o}$ has maximum volume among
all t.e.r.\ geodesic foliations $\mathcal{F}$ of $H$ such that the leaves of 
$\mathcal{F}$ and $\mathcal{F}_{o}$ intersecting the complement of some
compact subset of $B$ coincide.
\end{theorem}

\begin{remark}
For the \emph{three dimensional} hyperbolic space, the space of oriented
lines is a Kahler semi-Riemannian manifold \cite{Salvai-2007,Georgiou-2010} and $\mathcal{%
M}_{o}$ is a complex space-like surface. Then the assertion of the theorem
can also be proved using the Kahler form as a calibration \cite{Mealy}, the
semi-Riemannian version of the classical Wirtinger Theorem.
\end{remark}

\begin{proof} Let $\mathcal{F}$ be a t.e.r.\ geodesic foliation of $H$ as
in the statement of the theorem and let $\mathcal{M}$ be the corresponding
submanifold of $\mathcal{L}$. By Propositions \ref{F0sl} and \ref{rlt}, $%
\mathcal{M}_{o}$ and $\mathcal{M}$ are space-like.

Without loss of generality, we may take the compact subset of $B$ as being a
geodesic ball $B_{o}$ centered at $0$. Let $\mathcal{U}$ be the subset of $%
\mathcal{M}_{o}$ consisting of all the oriented geodesics in $\mathcal{M}%
_{o} $ intersecting $B$ in $B_{o}$. Then we can write $\mathcal{M}$ as $%
\mathcal{M}=\left( \mathcal{M}_{o}-\mathcal{U}\right) \cup \mathcal{V}$ for
some open space-like submanifold $\mathcal{V}$ of $\mathcal{M}$ of dimension 
$n$ with compact closure satisfying $\partial \mathcal{V}=\partial \mathcal{U%
}$.

By Definition \ref{hvm}~(a), we must prove that vol$_{n}\left( \mathcal{V}%
\right) \leq $ vol$_{n}\left( \mathcal{U}\right) $. This inequality will
follow from Theorem \ref{Teo1} provided that, according to Definition \ref%
{hvm}\ (b), we verify that $\mathcal{M}$ is contained in $\mathcal{L}%
^{\prime }$ and that $\mathcal{V}$ is homologous to $\mathcal{U}$.

Let us see first that $\mathcal{M}$ is contained in $\mathcal{L}^{\prime }$,
that is, $\varphi _{\pm }\left( \mathcal{M}\right) \subset S_{\pm }^{n}$.
Let $W$ be the unit vector field on $H$ which determines the foliation $%
\mathcal{F}$. The restriction of $W$ to $B$ coincides with $V$ on the
complement of $B_{o}$ and moreover cannot be tangent to $B$ at any point $x$. 
Otherwise, the geodesic with initial velocity $W\left( x\right) $, which
remains in $B$ ($B$ is totally geodesic), would intersect the geodesics of $%
\mathcal{F}$ intersecting the complement of $B_{o}$, a contradiction. By
continuity, $\left\langle V,W\right\rangle $ is a positive function on $B$.
Then the backward and forward Gauss maps carry the geodesics with initial
velocities $\left. W\right\vert _{B}$ to $S_{-}^{n}$ and $S_{+}^{n}$,
respectively. Thus, $\mathcal{M}\subset \mathcal{L}^{\prime }$.

Now we prove that $\mathcal{U}$ and $\mathcal{V}$ are homologous in $%
\mathcal{L}^{\prime }$. In order to do this, it will be convenient to
identify $\mathcal{L}^{\prime }$ with the product $B\times B$ and $\mathcal{M%
}_{o}$ and $\mathcal{M}$ with subsets of $B\times B$, which will turn out to
be graphs of certain maps on the ball, and then it will be clear that they
are homologous in the product.

Let $p_{\pm }:B\rightarrow S_{\pm }^{n}$ be defined by $p_{\pm }\left(
x\right) =\gamma _{V\left( x\right) }\left( \pm \infty \right) $, which is a 
diffeomorphism. We consider the following diagram%
\begin{equation*}
B\times B\overset{P}{\longrightarrow }S_{-}^{n}\times S_{+}^{n}\overset{F}{%
\longleftarrow }\mathcal{L}^{\prime }\text{,}
\end{equation*}%
where $F$ is defined in (\ref{isoF}) and $P:B\times B\rightarrow
S_{-}^{n}\times S_{+}^{n}$ is given by $P\left( x,y\right) =\left(
p_{-}\left( x\right) ,p_{+}\left( y\right) \right) $. Since both $F$ and $P\ 
$are diffeomorphisms, the map $P^{-1}\circ F$ provides an identification of $%
\mathcal{L}^{\prime }$ with $B\times B$. We have that $P^{-1}F\left( 
\mathcal{M}_{o}\right) $ is the graph of the identity on $B$, that is $%
\left\{ \left( x,x\right) \mid x\in B\right\} $.

Let us see that $P^{-1}\left( F\left( \mathcal{M}\right) \right) $ is the
graph of a certain diffeomorphism $f$ of $B$ with the property that $f$
coincides with the identity on the complement of $B_{o}$. Let $q_{\pm
}:B\rightarrow S_{\pm }^{n}$ be defined by 
\begin{equation*}
q_{\pm }\left( x\right) =\gamma _{W\left( x\right) }\left( \pm \infty
\right) ,
\end{equation*}%
which, by Proposition \ref{rltdifeo}, is a \emph{local} diffeomorphism
coinciding with $p_{\pm }$ on the complement of $B_{o}$. Then the map $%
f_{\pm }:=\left( p_{\pm }\right) ^{-1}\circ q_{\pm }:B\rightarrow B$ is a
local homeomorphism that coincides with the identity on the complement of $%
B_{o}$. By Lemma \ref{lemaBola}, $f_{\pm }$ is a homeomorphism.

Since $F\left( \mathcal{M}\right) =\left\{ \left( q_{-}\left( x\right)
,q_{+}\left( x\right) \right) \mid x\in B\right\} $, then 
$$
P^{-1}\left(F\left( \mathcal{M}\right) \right) =\left\{ \left( x,f\left( x\right)
\right) \mid x\in B\right\}, 
$$ 
where $f$ is the homeomorphism of $B$ given by $f=f_{+}\circ \left( f_{-}\right) ^{-1}$. 
Now the subsets  $\left\{ \left( x,x\right) \mid x\in B_{o}\right\}$ and $\left\{ \left( x,f\left(
x\right) \right) \mid x\in B_{o}\right\} $ are homologous in $B\times B$ and the proof
concludes.
\end{proof}

\bibliographystyle{amsplain}
\bibliography{mybib}

\end{document}